\documentclass[12pt]{amsart}
\usepackage{amsmath}
\usepackage{graphicx}
\usepackage{amsfonts}
\usepackage{amstext}
\usepackage{amsbsy}
\usepackage{amsopn}
\usepackage{amsxtra}
\usepackage{upref}
\usepackage{amsthm}
\usepackage{amsmath}
\usepackage{amssymb}

\newtheorem{thm}{Theorem}[section]
\newtheorem{cor}[thm]{Corollary}
\newtheorem{lem}[thm]{Lemma}

\theoremstyle{definition}
\newtheorem{defn}[thm]{Definition}
\theoremstyle{remark}
\newtheorem{rem}[thm]{Remark}
\newtheorem{ex}[thm]{Example}
\numberwithin{equation}{section}

\newcommand{\C}{\mathbb{C}}

\def\C{{\mathbb {C}}}
\def\B{{\mathbb B}}
\def\D{{\mathbb D}}

\def\S{{ S^{k}_{ij}}}

\begin{document}

\title {Prescribing the Preschwarzian in several complex variables.}

\author{Rodrigo Hern\'andez}
\thanks{The
author was partially supported by Fondecyt Grants \# 11070055.
\endgraf  {\sl Key words:} Schwarzian derivative, homeomorphic extension, ball, univalence,
convexity, Bergman metric.
\endgraf {\sl 2000 AMS Subject Classification}. Primary: 32H02, 32A17;\,
Secondary: 30C45.}
%

\begin{abstract}
We solve the several complex variables preSchwarzian operator
equation $[Df(z)]^{-1}D^2f(z)=A(z)$, $z\in \C^n$, where $A(z)$ is a
bilinear operator and $f$ is a $\C^n$ valued locally biholomorphic
function on a domain in $\C^n$. Then one can define a several
variables $f\to f_\alpha$ transform via the operator equation
$[Df_\alpha(z)]^{-1}D^2f_\alpha(z)=\alpha[Df(z)]^{-1}D^2f(z)$, and
thereby, study properties of $f_\alpha$. This is a natural
generalization of the one variable operator $f_\alpha(z)$ in
\cite{DSS66} and the study of its univalence properties, e.g., the
work of Royster \cite{Ro65} and many others. M\"{o}bius invariance
and the multivariables Schwarzian derivative operator of T. Oda
\cite{O} play a central role in this work.

\end{abstract}
\maketitle

\section{Introduction}
Consider the class $\mathcal S$ of functions $f$ holomorphic and
univalent in the disk $\D=\{z:|z|<1\}$ with the normalization
$f(0)=0$ and $f'(0)=1$.  Let $\alpha\in \C$, $f\in \mathcal S$ and
define the integral transform
\begin{equation}\label{f_alpha}f_\alpha(z)=\int_0^z\,[f'(w)]^{\alpha}\,dw\,,\end{equation}
where the power is defined by the branch of the logarithm for which
$\log f'(0)=0$, \cite{DSS66}. A question considered in \cite{DSS66}
is to determinate the values of $\alpha$ for which $f_\alpha \in
\mathcal S$. In \cite{Ro65} Royster exhibited non-univalent mappings
$f_\alpha$ for each complex $\alpha\neq 1$ with $|\alpha |>1/3$. In
fact, consider functions of the form
\begin{equation} \label{f_mu}f(z)=\exp(\mu\log(1-z))\,,\end{equation} which
are univalent if and only if $\mu$ lies in ones of the closed disks
$$|\mu+1|\leq 1\,,\;\;|\mu-1|\leq 1\,.$$ Royster showed that for any
such value of $\mu$, the function in (\ref{f_alpha}) is not
univalent for each $\alpha$ with$|\alpha|>1/3$ and $\alpha\neq 1$.
Moreover Pfaltzgraff, using the Ahlfors univalence criterion
\cite{Ah74},  proved that for any $f\in \mathcal{S}$, if
$|\alpha|\leq 1/4$ then $f_\alpha$ is univalent in
$\D$, see \cite{Pf75}. \\

Let $f$ be a locally univalent mapping in $\D$ and $f_\alpha$
defined by equation (\ref{f_alpha}) then $f_\alpha
'(z)=[f'(z)]^\alpha$, which implies that
$$\frac{f_\alpha''}{f_\alpha'}(z)=\alpha\frac{f''}{f'}(z)\,.$$ If
$f$ and $g$ satisfy that $g''/g'(z)=f''/f'(z)$ then
$\log(g'(z))=\log(f'(z))$ when $f'(0)=g'(0)$. Therefore $g=f$ if
$f(0)=g(0)$. Thus \begin{equation}
f_\alpha(z)=\int_0^z\,[f'(w)]^{\alpha}\,dw\;\;\Leftrightarrow\;\;
\frac{f_\alpha''}{f_\alpha'}(z)=\alpha\frac{f''}{f'}(z)\,.\end{equation}
This equivalence in one variable suggests our idea to define the
several variables generalization of $f_\alpha$ via operator equation
\begin{equation} [Df_\alpha(z)]^{-1}D^2f_\alpha(z)(\cdot,\cdot)=\alpha[Df(z)]^{-1}D^2f(z)(\cdot,\cdot)\, . \end{equation}

\noindent Yoshida developed, \cite{Y84}, a complete description of
prescribing Oda's Schwarzian derivatives \cite{O} in terms of a
completely integrable system of differential equations. The
description involves operators $\S f$ and $S^0_{ij}f$ of orders two
and three respectively, coefficients of the system and M\"{o}bius
invariants. In fact, the $\S f$ operators are the operator of least
order that vanish for M\"{o}bius mappings. This is a strong
difference with one complex variable where the third order
Schwarzian operator is the lowest order operator annihilated by all
M\"{o}bius mappings. For $n=1$, the M\"{o}bius group has dimension
$3$, which allows to set $f(z_0), f'(z_0)$ and $f''(z_0)$ of a
holomorphic mapping $f$ at a given point $z_0$ arbitrarily. It would
therefore be pointless to seek a M\"{o}bius invariant differential
operator of order 2. But for $n>1$ the number of parameters involved
in the value and all derivatives of order 1 and 2 of a locally
biholomorphic mapping is $n^2(n+1)/2+n^2+n$, and exceeds the
dimension of the corresponding M\"{o}bius group in $\C^n$, which is
$n^2+2n$. By the definition of the Schwarzian derivatives, we have
that $S^k_{ij}F=S^k_{ji}F$ for all $k$ and $\sum_{j=1}^nS^j_{ij}F=0$
and we see there are exactly $n(n-1)(n+2)/2$ independent terms
$S^k_{ij}F$, which is equal to the excess mentioned above.

A different approach to obtain the invariant operators $S_{ij}^k,\,
S_{ij}^0$ has been developed by Molzon and Tamanoi (\cite{MT}). In
addition, Molzon and Pinney had earlier developed
equivalent invariant operators in the context of complex manifolds (\cite{MP}). \\

The operator
$$P_f(z)= [Df(z)]^{-1}D^2f(z)(\cdot,\cdot)$$ introduced by Pfaltzgraff in \cite{Pf74} is the ``natural '' way to extend
the classical one variable operator preSchwarzian $f''/f' $.
Furthermore, the author in \cite{Pf74} extended the classical
univalence criterion of Becker, (\cite{B72}), to several variables.
The question now is how to extend the equation (\ref{f_alpha}) to
$\C^n$. It is necessary to understand when one can recover the
function $f$ from a given $P_f$. We shall show a strong connection
between this operator and the Schwarzian derivatives operator
$SF(z)(\cdot,\cdot)$, introduced in \cite{RH1}.  Indeed, the problem
of prescribing $P_f$ can be reduced to understanding how to
prescribe $\S f$ in terms of $P_f$. This is achieved via completely
integrable system generated by $\S f$ and corresponding ``new
differential conditions'' on the elements of $P_f$. We then use this
theory to extend the classical single variable problem about the
univalence of $f_\alpha$ by using equation (1.4) to define
$f_\alpha$ in several complex variables.

\section{Oda Schwarzian and M\"{o}bius Invariants}

Let $f:\Omega\subset\C^n\rightarrow\C^n$ be a locally biholomorphic
mapping defined on some domain $\Omega$. T.Oda in \cite{O} defined
the Schwarzian derivatives of $f=(f_1,\ldots,f_n)$ as
\begin{equation}\label{schwarzian-def}\S f= \displaystyle
\sum^{n}_{l=1}\frac{\partial^{2} f_{l}}{\partial z_{i}\partial
z_{j}} \frac{\partial z_{k}}{\partial
f_{l}}-\frac{1}{n+1}\left(\delta^{k}_{i}\frac{\partial}{\partial
z_{j}}+\delta^{k}_{j}\frac{\partial}{\partial z_{i}}\right) \log
J_f\, ,\end{equation} where $i,j,k=1,2, \ldots,n,$ $J_f$ is the
jacobian determinant of the diferential $Df$ and $\delta^{k}_{i}$
are the Kronecker symbols. For $n>1$ the Schwarzian derivatives have
the following properties:
\begin{equation}\label{mobius-propertie}\S f=0 \;\;\; \mbox{for
all}\quad i,j,k=1,2,\ldots ,n
 \;\;\;\mbox{iff}\;\;\;f(z)= M(z)\, ,\end{equation} for some
M\"{o}bius transformation
$$M(z)=\left(\frac{l_1(z)}{l_0(z)},\ldots,\frac{l_n(z)}{l_0(z)}\right)\, ,$$
where $l_i(z)=a_{i0}+a_{i1}z_1+\cdots+a_{in}z_n$ with
$\det(a_{ij})\neq 0$. Furthermore, for a composition

\begin{equation}\label{cadena}\S(g\circ f)(z)=\S f(z)+ \sum^{n}_{l,m,r=1}
S^{r}_{lm}g(w)\frac{\partial w_{l}} {\partial z_{i}}\frac{\partial
w_{m}} {\partial z_{j}} \frac{\partial z_{k}} {\partial w_{r}}\; ,
\; w=f(z) \, .\end{equation} From this chain rule it can be shown
that $S^k_{ij}f= S^k_{ij}g$ for all $i,j,k=1,\ldots,n$ if and only
if $g=T\circ f$ for some M\"{o}bius transformation. The $S^0_{ij}f$
coefficients are given by
$$S^0_{ij}f(z)=J_f^{1/(n+1)}\left(\frac{\partial^2}{\partial z_i\partial z_j}
J_f^{-1/(n+1)}-\sum_{k=1}^n\,\frac{\partial}{\partial
z_k}J_f^{-1/(n+1)}S^k_{ij}f(z)\right)\, .$$

In his work, Oda gives a description of the functions with
prescribed Schwarzian derivatives $\S f$ (\cite{O}). Consider the
following overdetermined system of partial differential equations,
\begin{equation}\label{sistema}
\frac{\partial^{2}u}{\partial z_{i}\partial z_{j}} =
\sum^{n}_{k=1}P^{k}_{ij}(z)\frac{\partial u}{\partial z_{k}}+
P^{0}_{ij}(z)u \; , \quad i,j= 1,2, \ldots,n\, ,\end{equation} where
$z=(z_{1},z_{2},...,z_n)\in \Omega\subset\C^{n}$ and $P^k_{ij}(z)$
are holomorphic functions for $i,j,k=0,\ldots,n$. The system
(\ref{sistema})  is called {\it completely integrable} if there are
at most $n+1$ linearly independent solutions, and is said to be in
{\it canonical form} (see \cite{Y76}) if the coefficients satisfy
$$\sum_{j=1}^{n} P_{ij}^{j}(z)=0 \; ,\quad i=1,2,\ldots,n.$$ T.
Oda proved that (\ref{sistema}) is a completely integrable system in
canonical form if and only if $P^k_{ij}=\S f$ for a locally
biholomorphic mapping $f=(f_1, \ldots, f_n)$, where $f_i=u_i/u_0$
for $1\leq i\leq n$ and $u_0, u_1,\ldots, u_n$ is a set of linearly
independent solutions of the system. For a given mapping $f$,
$u=\left(J_f\right)^{-\frac{1}{n+1}}$ is always a solution of
(\ref{sistema}) with $P_{ij}^k=\S f$.

\begin{defn} We define the \textit{Schwarzian derivative operator} as the
operator $S_f(z):T_z\Omega \to T_{f(z)}\Omega$ given by
$$S_ f(z)(\vec{v},\vec w)=\left(\, \vec{v}^{\,t}S^{1}f(z)\vec{w}\, ,
\,\ldots,\vec{v}^{\,t} S^{n}f(z)\vec{w}\, \right)\, ,$$ where $
S^{k} f$ is the $n\times n$ matrix defined by  $(\S f)_{ij}$ and
$\vec{v}\in T_z \Omega$.
\end{defn}

The Schwarzian derivative operator \cite{RH2} can be rewritten as
\begin{equation}\label{free-coord-u_0}S_f(z)(\vec v,\vec w)=[Df(z)]^{-1}D^2f(z)(\vec v,\vec
w)-\frac {1}{n+1}\left(\nabla \log J_f(z)\cdot \vec v\right)\vec
w-\frac{1}{n+1}\left(\nabla \log J_f(z)\cdot \vec w\right)\vec
v\,,\end{equation} and the system ($\ref{sistema}$) as

\begin{equation}\label{vectorial} \mbox{Hess} \,u(z)(\cdot,\cdot)=\nabla u(z)\cdot S_f(z)(\cdot,\cdot)+
S^0 _f(z)(\cdot,\cdot)u(z)\,,\end{equation} where $S^0 _f$ is a
$n\times n$ matrix defined by $(S^0_{ij}f)_{ij}$. We include in this
section two lemmas that complement the work of Oda.

\begin{lem}Let $f:\Omega\subset \C^n\to \C^n$ be a locally biholomorphic mapping and $u_0=J_f^{-1/n+1}$,
then
$$f=\frac{\vec u}{u_0}=\left(\frac{u_1}{u_0},\ldots,\frac{u_n}{u_0}\right)\,,$$ where $u_0,u_1,\ldots,u_n$ are linearly independent solutions of (\ref{sistema})
\end{lem}

\begin{proof} We will prove that $\vec u=fu_0$ is solution of the equation (\ref{vectorial}). It follows that $Dfu_0+f\nabla u_0=Du$, from where $$D^2f\cdot u_0+2Df\cdot\nabla u_0+f\cdot\mbox{Hess}\, u_0=D^2u\,.$$ Using the system we have that
$$D^2f\cdot u_0+2Df\cdot \nabla u_0-Df\cdot u_0\cdot S_f+Du\cdot S_f+S^0_f\cdot u=D^2 u\,.$$ Considering the equation (\ref{free-coord-u_0}) with $u_0=J_f^{-1/n+1}$ we have that $$D^2f\cdot u_0+2Df\cdot \nabla u_0-Df\cdot Sf\cdot u_0=0\,,$$ and $D^2u(\cdot,\cdot)=Du(S_f(\cdot,\cdot))+S^0_f(\cdot,\cdot)u$, hence $u_i$ with $i=1,\ldots,n$ and $u_0$ are independent solutions of the system (\ref{sistema}).

\end{proof}

\begin{lem}\label{2.3} Let $u_0$ be a solution of the system (\ref{sistema}). Then there exists a function $f=\vec u/u_0$
where $\vec u=(u_1,\ldots,u_n)$ and $u_i$ with $i=0,1,\ldots,n$ are
independent solutions of the system (\ref{sistema}) where
$u_0=J_f^{-1/n+1}$. The function $f$ will be holomorphic away from
the zero set of $u_0$.
\end{lem}

\begin{proof} According to previous lemma we can find $F=\vec v/v_0$ where $\{v_0,v_1,\ldots,v_n\}$ are a linearly independent solutions of the system (\ref{sistema}) with $P^k_{ij}=S^k_{ij}$ and $v_0=J_F^{-1/n+1}$. As $u_0$ is solution of the system we have that $u_0=\alpha_0v_0+\cdots+\alpha_nv_n$.  We need to find a M\"obius mapping $T$ such that
$$T\circ F=\left(\frac{u_1}{u_0},\ldots,\frac{u_n}{u_0}\right)=f\,,$$ and $J_{T\circ F}^{-1/n+1}=u_0$. We have
$$\begin{array}{ccl} J_{T\circ F}^{-1/n+1}(z) & = & J_T^{-1/n+1}(F(z))J_F^{-1/n+1}(z)\\[0.3cm]
 & = & (\lambda_0+\lambda_1F_1(z)+\cdots+\lambda_nF_n(z))J_F^{-1/n+1}(z)\\[0.3cm]
  & = & \lambda_0v_0+\lambda_1v_1+\cdots+\lambda_nv_n\,, \end{array}$$ which will be
  equal to $u_0$ if we choose $\lambda_i=\alpha_i$ for all $i=0,1,\ldots,n\,.$

\end{proof}

 \section{Results}
Let $\Omega\subset \C^n$ be domain.

\begin{thm} Let $f:\Omega\to\C^n$ be a locally biholomorphic mapping. The following statements are equivalent:\\
\begin{enumerate}
\item[(i)] $S^0_{ij}f(z)\equiv 0$.\\
\item[(ii)] There exists a locally biholomorphic mapping $g:\Omega\to\C^n$  with $S_g=S_f$ and $J_g$ constant.\\
\item[(iii)] There exists a locally biholomorphic mapping $h:\Omega\to\C^n$  such that $S_h=S_f$ and $J_h^{-1/n+1}=1/L(h)$, where $L(w)=\alpha_0+\alpha_1w_1+\cdots+\alpha_nw_n$.\\
\item[(iv)] Locally there exists a biholomorphic change of variables such that the system (\ref{sistema}) with $P^k_{ij}=\S f$ reduces to $\mbox{Hess}(u)=0$.\\
\end{enumerate}
\end{thm}

\begin{proof} $(i)\rightarrow (ii)$. As $S^0_{ij}f\equiv 0$, the system (\ref{sistema}) reduces to $$u_{ij}=\sum_{k=1}^n\S u_k\,.$$
Therefore $u\equiv c$ is solution, thus by Lemma (\ref{2.3}) there exists a function $g$ such that $J_g\equiv C$.\\

$(ii)\to (iii)$. Let $g=T\circ h$ for some M\"{o}bius $T$ to be
determined. Then
$J_g^{-1/n+1}(z)=J_T^{-1/n+1}(h(z))J_h^{-1/n+1}(z)$. Since
$J_g^{-1/n+1}\equiv C$ we have that
$$C=(a_0+a_1h_1+\cdots+a_nh_n)J_h^{-1/n+1}(z)\,,$$ from where the result obtains after scaling $h$.\\

$(iii)\to (iv)$. Suppose $h$ has $J_h^{-1/n+1}=1/L(h)$. The previous
argument shows that by choosing $T$ appropirately, we can produce
$g=T\circ h$ with  $J_g\equiv 1$. Hence
$S_g(z)(\cdot,\cdot)=(Dg(z))^{-1}D^2g(z)(\cdot,\cdot)$, and the
system (\ref{sistema}) reduces to $$\mbox{Hess
}u(z)(\cdot,\cdot)=\nabla u(z)\cdot S_g(z)(\cdot,\cdot)\,.$$ We
consider $D(\nabla u(z)(Dg(z))^{-1})(\cdot,\cdot)$:
$$\begin{array}{cll}D(\nabla u(z)(Dg(z))^{-1}(\cdot,\cdot)& = &\mbox{Hess }u(z)((Dg(z))^{-1}(\cdot),\cdot)\\[0.2cm]
& & -\nabla u(z)\cdot
(Dg(z))^{-1}D^2g(z)((Dg(z))^{-1}(\cdot),\cdot)\\ [0.2cm]
 &= & \nabla u(z)\cdot S_g(z)((Dg(z))^{-1}(\cdot),\cdot)\\ [0.2cm]
 & & -\nabla u(z)\cdot (Dg(z))^{-1}D^2g(z)((Dg(z))^{-1}(\cdot),\cdot)=0\,.\end{array}$$

Let $\varphi$ a local inverse of $g$. Therefore $U(w)=u(\varphi(w))$ satisfies that $\nabla U=\nabla u\cdot D\varphi=\nabla u(z)(Dg(z))^{-1}$, thus $\mbox{Hess }U(w)\equiv 0$. \\

$(iv)\to (i)$. Since $\mbox{Hess } u(s)\equiv 0$, then $u\equiv c$
is a solution of this system (\ref{sistema}) therefore
$S^0_{ij}f\equiv 0$.
\end{proof}

\begin{thm} Let $f:\Omega\to\C^n$ be a locally biholomorphic mapping. There exists a function $g:\Omega\to\C^n$
locally biholomorphic such that \begin{equation}
Dg(z)=Df(z)J_f^{-\frac{2}{n+1}}\,\end{equation} if and only if
$S^0_{ij}f\equiv 0$ for all $i$ and $j$. The function $g$ will have
$S_g=S_f$.
\end{thm}

\begin{proof} Suppose (3.1) holds. A straightforward calculation shows that $$(Dg(z))^{-1}D^2g(z)(v,v)=S_f(z)(v,v)\,.$$
The coordinate functions $g^{i}$ of function $g$ satisfy
$$dg^{i}=J_f^{-2/n+1}df^{i}\,. $$ Since $0=d^2g^{i}=d^2f^{i}$ we conclude that
$J_f$ must be a constant. By Theorem 3.1 we conclude that
$S^0_{ij}f\equiv 0$ for all $i$ and $j$. Reciprocally, if
$S^0_{ij}f\equiv 0$ then there exists a constant solution of the
system (2.4), and by Lemma 2.2 there exists a mapping $g$ with
$S_g=S_f$ and $J_g^{-1/n+1}\equiv C$.  By (2.5), $S_g=P_g=S_f$.
\end{proof}

\begin{rem}Considering $S^0_{ij}f\equiv 0$ then $cDf=Dg$ for some constant $c$. When $c=J_f^{-2/n+1}$ we have that
$$P_g(z)=S_f(z)=P_f(z)\,.$$
\end{rem}

M.A. Goldberg in \cite{Gold} showed that, in terms of the our
operator,
\begin{equation}\label{Goldberg}  tr\,\{Df(z)^{-1}D^2f(z)(\vec v_i,\cdot)\}=\frac{\partial }{\partial z_i}\,\log J_f(z) \,,\end{equation}
where $\vec v_i=(0,\ldots,1,\ldots,0)$ with 1 is in position $i$. We
use this result to prove the next theorem of uniqueness.

\begin{thm} Let $f,g$ be locally biholomorphic mappings defined in $\Omega$. Then $P_f(z)(\cdot,\cdot)=P_g(z)(\cdot,\cdot)$ if and only if $f=T\circ g$, where
$T(z)=Az+b$ with $A$ is a $n\times n$ constant matrix and $b\in
\C^n$.
\end{thm}

\begin{proof}  Let $f$ and $g$ be locally biholomorphic mappings in $\Omega$. As $P_f(z)(\vec v_i,\cdot)=P_g(z)(\vec v_i,\cdot)$  for all $i=1,\ldots,n$ then by equation (\ref{Goldberg}) we have that
\begin{equation}\label{3.2}\nabla \log J_f(z)=\nabla \log J_g(z)\,.\end{equation} Using equation (\ref{free-coord-u_0}) we can conclude that $S_f(z)=S_g(z)$ for all $z$. Hence $g=T\circ f$ for some Mobius mapping $T$. But $\log J_g(z)=\log J_T(f(z))+\log J_f(z)$ and equation (\ref{3.2}) we have that $\log J_T(z)$ is a constant, therefore $T(z)=Az+b$ for some $n\times n$ matrix $A$ and $b\in \C^n$. Reciprocally, if $f=T\circ g$ with $T(z)=Az+b$ for some $n\times n$ matrix $A$ and $b\in \C^n$, is easy to see that $Df(z)=DT(f(z))Df(z)=ADf(z)$, which implies that $P_f(z)=P_g(z)$.
\end{proof}

\begin{thm} Let $A(z)$ be a bilinear operator defined in $\Omega$ by $$A(z)(\vec v,\cdot)= \left(
\begin{array}{cccccc}
a^1_{11}v_1+\cdots +a^1_{1n}v_n & \cdot  & \cdot & \cdot &
\cdot & a^1_{n1}v_1+\cdots +a^1_{nn}v_n  \\
\cdot & \cdot & \cdot & \cdot &
\cdot & \cdot \\
\cdot & \cdot & \cdot & \cdot & \cdot & \cdot \\
\cdot & \cdot & \cdot & \cdot & \cdot & \cdot \\
a^n_{11}v_1+\cdots +a^n_{1n}v_n  & \cdot & \cdot & \cdot
& \cdot & a^n_{n1}v_1+\cdots +a^n_{nn}v_n %
\end{array}%
\right)$$ where $a^k_{ij}=a^k_{ij}(z)$ and $\vec v=(v_1\ldots,v_n)$.
Then there exists a function $f:\Omega\to\C^n$ locally biholomorphic
such that $P_f(z)=A(z)$ if and only if the following statements
hold:
\begin{enumerate}
\item [(i)] $a^k_{ij}(z)=a^k_{ji}(z)$ for all $i,j,k=1,\ldots,n.$;\\

\item [(ii)] there exists a  holomorphic function $\varphi:\Omega\to\C$ such that $$a^1_{1j}(z)+a^2_{2j}(z)+\cdots+a^n_{nj}(z)=
\frac{\partial \varphi}{\partial z_j}(z)\;\;\; \forall
j=1,\ldots,n\,;$$

\item [(iii)]$\exp(-\displaystyle\frac{\varphi}{n+1})$ is a solution of the system (\ref{sistema}) with $P^k_{ij}(z)$ given by
$$P^k_{ij}(z)=a^k_{ij}(z)-\frac{1}{n+1}\left(\delta^k_i\, tr\,\{A(z)(\vec v_j,\cdot)+\delta^k_j\, tr\,\{A(z)(\vec v_i,\cdot)\right)\,\, i,j,k=1,\ldots,n,$$ and $P^0_{ij}(z)$ are defined in terms of $P^k_{ij}(z)$ such that the integrable condition of the system (\cite{Y84} page 129-130) holds.
\end{enumerate}
\end{thm}

\begin{proof} Using $(i)$ and $(ii)$ we have that
$$\mbox{tr}\{A(z)(\lambda,\cdot)\}=\nabla \varphi(z)\cdot \lambda\,.$$ For given $A(z)$ we can construct a bilinear mapping $\Lambda(z)(\lambda,\mu)$ as $$\Lambda(z)(\lambda,\mu)=A(z)(\lambda,\mu)-\frac{1}{n+1}\mbox{tr}\{A(z)(\lambda,\cdot)\}\mu-\frac{1}{n+1}\mbox{tr}\{A(z)(\mu,\cdot)\}\lambda\,.$$ Each component of $\Lambda(z)$ is $P^k_{ij}$ defined by
$$a^k_{ij}(z)-\frac{1}{n+1}\left(\delta^k_i\, \mbox{tr}\,\{A(z)(\vec v_j,\cdot)+\delta^k_j\, \mbox{tr}\,\{A(z)(\vec v_i,\cdot)\right)\,.$$ These coefficients satisfy $\sum_i P^k_{ik}=0$ for all $k=1,\ldots,n$. Now we define coefficients $P^0_{ij}$ in terms of $P^k_{ij}$ with $k=1,\ldots,n$ such that the integrability conditions in \cite{Y84} hold, (see pages 129-130). Thus, the system (\ref{sistema}) is completely integrable and in canonical form. Hence we can construct a function $f$ such that $S_f(z)=\Lambda(z)$. By $(iii)$ we have that
$$J_f^{-1/n+1}=\exp(-\displaystyle\frac{\varphi}{n+1})\,.$$ As $S_f$ is defined by equation (\ref{free-coord-u_0}) we conclude that $$\mbox{tr}\,\{A(z)(\lambda,\cdot)\}=\frac{1}{n+1}\nabla J_f(z)\cdot\lambda\,,$$ which implies that
$$P_f(z)=(Df(z))^{-1}D^2f(z)(\cdot,\cdot)=A(z)(\cdot,\cdot)\,.$$ Reciprocally, is easy to see that $P_f(z)$ satisfies (i), (ii) and (iii).
\end{proof}
Observe that $\alpha[Df(z)]^{-1}D^2f(z)(\vec v,\cdot)$ by locally
biholomorphic function $f$, satisfies $(i), (ii)$ and $(iii)$ of the
Theorem 3.4.
\begin{defn} Let $f$ be a locally biholomorphic mapping in $\Omega$ such that $f(0)=0$ and $Df(0)=Id$. We define
$f_\alpha$ in $\Omega$ as the locally biholomorphic mapping for
which
\begin{equation}\label{dfd^2falpha}[Df_\alpha(z)]^{-1}D^2f_\alpha(z)(\cdot,\cdot)=\alpha[Df(z)]^{-1}D^2f(z)(\cdot,\cdot)\,,\end{equation}
and $f_\alpha(0)=0$, $Df_\alpha(0)=Id$.
\end{defn}

As a generalization of the problem raised in \cite{DSS66}, one can
ask the question of determining the values of $\alpha$ for which the
mapping $f_{\alpha}$ is univalent when $f$ is univalent or even just
locally univalent. A partial answer is given below when $f$ is
convex in the unit ball $\B^n$. Theorem 3.5 shows another partial
result for compact linear invariant families. Since the class of
univalent mappings in $\B^n$ fails to be compact ($n>1$), we think
it is unlikely that there exists an $\alpha_0>0$ small enough so
that $f_{\alpha}$ is univalent for any $|\alpha|\leq \alpha_0$ and
$f$ univalent in $\B^n$. An interesting compact family of univalent
mappings to consider would be the class $S_0$ of univalent mappings
in $\B^n$ that have a parametric representation.


\begin{ex}
\noindent Let $f(z_1,z_2)=(\phi_\alpha(z_1),\psi_\alpha(z_2))$ be a
locally univalent mapping defined in $\B^2$ such that
$\phi_\alpha(z_1)$ and $\psi_\alpha(z_2)$ are defined by the
equation (\ref{f_alpha}) where $\phi$ and $\psi$ are locally
univalent analytic mappings defined in the unit disc such that
$\phi(0)=\psi(0)=0$, $\phi' (0)=\psi'(0)=1$ and suppose that
$z=(z_1,z_2)\in \B^2$. Its Schwarzian derivatives satisfy
$$\begin{array}{cll}S^1_{11}f(z_1,z_2) & = & \displaystyle\frac{\phi_\alpha''}{\phi_\alpha'}(z_1)=\alpha \frac{\phi''}{\phi'}(z_1),\\ [0,4cm]
S^2_{22}f(z_1,z_2) & = &
\displaystyle\frac{\psi_\alpha''}{\psi_\alpha'}(z_2)=\alpha
\frac{\psi''}{\psi'}(z_2),\\ [0,4cm] S^1_{22}f(z_1,z_2) & = &
S^2_{11}f(z_1,z_2)=0\,.\end{array}$$ Now, let
$f(z)=(\psi(z_1),\phi(z_2))$. Then the corresponding mapping
$f_\alpha$ has the property that its Schwarzian derivatives are
$$\begin{array}{ccl}S^1_{11}f_\alpha(z_1,z_2)&=&\alpha S^1_{11}f(z_1,z_2)=\displaystyle\alpha \frac{\phi''}{\phi '}(z_1)\,,\\ [0.3cm]
S^2_{22}f_\alpha(z_1,z_2)& =& \alpha
S^2_{22}f(z_1,z_2)=\displaystyle\alpha
\frac{\psi''}{\psi'}(z_2)\,,\\ [0.3cm] S^1_{22}f_\alpha(z_1,z_2)& =
& S^2_{11}f_\alpha(z_1,z_2)=0\,.\end{array}$$ Therefore $\S f = \S
f_\alpha$ which implies that there exists a M\"{o}bius mapping $M$
such that $M\circ f=f_\alpha$. But $f(0)=0=f_\alpha(0)$,
$DF(0)=Id=Df_\alpha(0)$ and $\nabla \log J_f=\nabla \log
J_{f_\alpha}=\alpha \nabla \log J_f$, then $f=f_\alpha$. Thus
$$f(z)=(\phi(z_1),\psi(z_2))\Longrightarrow f_\alpha(z)=(\phi_\alpha(z_1),\psi_\alpha(z_2))\,,$$ where $\phi_\alpha$ and $\psi_\alpha$ are defined by (\ref{f_alpha}). By the way, in this example if $|\alpha|<1/4$ then $f_\alpha$ will be univalent in $\B^2$.  Moreover if $\phi(z_1)$ is a univalent mapping defined by (\ref{f_mu}) and $\psi(z_2)=z_2$, then the mapping $f(z)=(\phi(z_1),\psi(z_2))$ is univalent and the  corresponding mapping $f_\alpha$ is not univalent if $|\alpha|>1/3$ and $\alpha\neq 1$.
\end{ex}
In \cite{Kik} the author proved that $f:\B^n\to\C^n$ locally
biholomorphic mapping is convex if and only if
$1-\mbox{Re}\langle[Df(z)]^{-1}D^2f(z)(u,u),z\rangle>0$ for all
$z\in\B^n$ and $u\in\C^n$ with $\|u\|=1$. Thus, if $0\leq \alpha\leq
1$ then $f_\alpha$ is a convex mapping when $f$ is a convex mapping
since
\begin{equation}\label{convex-alpha}1-\mbox{Re}\langle[Df_\alpha(z)]^{-1}D^2f_\alpha(z)(u,u),z\rangle=1-\alpha
\mbox{Re}\langle[Df(z)]^{-1}D^2f(z)(u,u),z\rangle
>0\,.\end{equation}

\begin{ex} Let $f$ be a univalent function in $\D$. We consider the Roper-Suffridge extension (see\cite{RS}) to $\B^2$ of $f$ to the function
$$\Phi_f(z)=\left(f(z_1),\sqrt{f'(z_1)}z_2\right)\,.$$ Thus, $$[D\Phi_f(z)]^{-1}[D^2\Phi_f(z)](\vec v,\cdot)=\begin{pmatrix}
    \frac{f''}{f'}(z_1)v_1  & 0   \\
     \frac 12z_2Sf(z_1)v_1+\frac 12\frac{f''}{f'}(z_1)v_2 &  \frac 12\frac{f''}{f'}(z_1)v_1\end{pmatrix}\,.$$ A straightforward  calculation shows that $$(\Phi_f)_\alpha(z)=\left(f_\alpha(z_1),z_2\sqrt{f'_\alpha(z_1)}+y(z_1)\right)\,,$$ where $f_\alpha$ is defined by equation (\ref{f_alpha}) and $y$ satisfies that $$y''-\alpha\frac{f''}{f'}y'=\frac{\alpha(\alpha-1)}{4}\left(\frac{f''}{f'}\right)^2(f')^{\alpha/2}\,.$$ Moreover $\Phi_f$ is univalent when $f$ is univalent, in fact if $f$ is convex then $\Phi_f$ is convex. On the other hand, $(\Phi_f)_\alpha$ is univalent if $f_\alpha$ is univalent which holds for $|\alpha|\leq 1/4$ for all univalent mappings $f$.
\end{ex}

\begin{thm} Let $f:\B^n\to \C^n$ be a locally biholomorphic mapping such that the norm order of the linear invariant family generated by $f$ is $\beta<\infty$. Then $f_\alpha$ is univalent if $|\alpha|\leq \displaystyle\frac{1}{2\beta+1}$.
\end{thm}

\begin{proof}  Let $\phi$ be a automorphism of $\B^n$ such that $\phi(0)=\zeta$.
The mapping
$g(z)=D\phi(0)^{-1}Df(\phi(0))^{-1}(f(\phi(z))-f(\phi(0)))$ belongs
to the family generated by $f$, therefore $\|D^2g(0)\|\leq \beta$.
But $$
\begin{array}{cll}D^2g(0)(\cdot,\cdot)& = & D\phi(0)^{-1}Df(\zeta)^{-1}Df(w)D^2\phi(0)(\cdot,\cdot)\,+\\
& &
D\phi(0)^{-1}Df(\zeta)^{-1}D^2f(\zeta)(D\phi(0)(\cdot),D\phi(0)(\cdot))\,.\end{array}$$
Evaluating in $D\phi(0)^{-1}(\zeta)=\zeta /(1-\|\zeta\|^2)$,
multiplication by $\alpha$ and using (\ref{dfd^2falpha}) we have
that
$$ \begin{array}{cll}\alpha D^2g(0)(\zeta,\cdot)& = & \alpha D\phi(0)^{-1}Df(\zeta)^{-1}Df(\zeta)D^2\phi(0)(\zeta,\cdot)\,+\\
& &
(1-\|\zeta\|^2)D\phi(0)^{-1}Df_\alpha(\zeta)^{-1}D^2f_\alpha(\zeta)(\zeta,D\phi(0)(\cdot))\,,\end{array}$$
where
$D\phi(0)^{-1}D^2\phi(\zeta,\cdot)=-\|\zeta\|^2(\cdot)-\zeta\zeta^{*}(\cdot)$.
Thus, for all vectors $v=D\phi(0)^{-1}(u)$ it follows that
$$\begin{array}{lll} (1-\|\zeta\|^2)Df_\alpha(\zeta)^{-1}D^2f_\alpha(\zeta)(\zeta,u)& = & \alpha D\phi(0)D^2g(0)(\zeta,v)-\alpha \|\zeta\|^2u\\
& & -\alpha(1-\|\zeta\|^2)\zeta\zeta^{*}v\,.\end{array}$$ Then
taking supremum over all vectors $u$ with norm $\|u\|=1$, we have
that
$$\|(1-\|\zeta\|^2)Df_\alpha(\zeta)^{-1}D^2f_\alpha(\zeta)(\zeta,\cdot)+\alpha\|\zeta\|^2I\|\leq
|\alpha|(2\beta+1)\,,$$ hence by the generalization of Ahlfors and
Becker result (see page 350 in \cite{GK}) we conclude that
$f_\alpha$ satisfies the hypothesis of this theorem so is univalent
in $\B^n$.

\end{proof}

The last corollary is an immediate consequence.

\begin{cor} Let $\mathcal{F}$ be a linearly invariant family of locally
biholomorphic mappings defined in $\B^n$ of finite order $\beta$.
Then $f_{\alpha}$ is univalent in $\B^n$ for every $f\in\mathcal{F}$
and $|\alpha|\leq \displaystyle\frac{1}{2\beta+1}$.
\end{cor}


\noindent Acknowledgements: I would like to thank Martin Chuaqui and
John Pfaltzgraff for their useful suggestions and valuable
discussions.

\bibliographystyle{amsplain}

\medskip
\noindent Facultad de Ciencias y Tecnolog\'ia, Universidad Adolfo
Ib\'a\~nez, Av. Balmaceda 1625 Recreo, Vi\~na del Mar, Chile,\,
\email{rodrigo.hernandez@uai.cl}

\end{document}